\documentclass{article}
\usepackage[utf8]{inputenc}
\usepackage{graphicx}
\usepackage{amssymb,amsthm,amsmath,comment}
\usepackage{hyperref}
\usepackage{fancyhdr}
\usepackage{xurl}
\usepackage[left=3cm,right=3cm]{geometry}
\usepackage[T1]{fontenc}

\newcommand{\Z}{{\mathbb Z}}
\newcommand{\F}{{\mathbb F}}
\newcommand{\Q}{{\mathbb Q}}
\newcommand{\R}{{\mathbb R}}
\newcommand{\myell}{l}  
\newcommand{\GL}{\operatorname{GL}}
\newcommand{\eps}{\varepsilon}  
\newcommand{\myleq}{\leqslant}  
\newcommand{\mygeq}{\geqslant}  

\newtheorem{theorem}{Theorem}[section]
\newtheorem{lemma}[theorem]{Lemma}
\newtheorem{proposition}[theorem]{Proposition}
\newtheorem{corollary}[theorem]{Corollary}

\theoremstyle{definition}
\newtheorem{definition}[theorem]{Definition}
\newtheorem{definitionlemma}[theorem]{Definition/Lemma}
\newtheorem{remark}[theorem]{Remark}

\newenvironment{proofof}[1]{\par\noindent{\em Proof of #1.}}%
                        {\hspace*{\fill}\nobreak$\Box$\par\medskip}

\title{The density of integral quadratic forms having a \\
        $k$-dimensional totally isotropic subspace}
\author{Lycka Drakengren and Tom Fisher}
\date{20th December 2022}

\begin{document}
\maketitle

\begin{abstract}
  We investigate the probability that a random quadratic form
  in $\Z[x_1,...,x_n]$ has a totally isotropic subspace of a
  given dimension. We show that this global probability is a
  product of local probabilities. Our main result computes
  these local probabilities for quadratic forms over the $p$-adics.
  The formulae we obtain are rational functions in $p$ invariant
  upon substituting $p \mapsto 1/p$.
\end{abstract}

\vspace{8pt}

\section{Introduction}

An integral quadratic form in $n$ variables is a homogeneous
polynomial of degree $2$
\begin{equation}\label{Q}
    Q(x_1,x_2,\ldots,x_n) = \sum_{1\myleq i \myleq j \myleq n}a_{ij}x_ix_j
\end{equation}
where the coefficients $a_{ij}$ belong to $\Z$.  A non-zero vector
$(x_1,x_2, \ldots ,x_n) \in V = \Q^n$ is called \textit{isotropic} if
$Q(x_1,x_2,\ldots,x_n) = 0$. A subspace of $V$ is {\em totally
  isotropic} if all its non-zero vectors are isotropic. We say that $Q$ is
{\em $k$-isotropic} if $V$ has a $k$-dimensional totally isotropic
subspace. A quadratic form that is $1$-isotropic is simply called
isotropic.

Generalising the results of \cite{BCFJK}, where only the case $k=1$
was considered, we investigate the probability $\rho_{\rm glob}(k,n)$
that a random integral quadratic form in $n$ variables is
$k$-isotropic. More formally we define
\begin{equation}
\label{def:global}
\rho_{\rm glob}(k,n) = \lim_{H \to \infty}
\frac{ \# \left\{ \begin{array}{c}
  \text{  quadratic forms $Q = \textstyle\sum a_{ij} x_i x_j
  \in \Z[x_1, \ldots, x_n]$} \\ \text{ 
  with $|a_{ij}| \myleq H$ that are $k$-isotropic over $\Q$ }
\end{array} \right\} }{(2H)^{n(n+1)/2}} 
\end{equation}
if this limit exists.

Combining the Strong Hasse Principle \cite[p. 75]{Cassels-RQF} and
Witt's Cancellation Theorem (see Theorem~\ref{Witt}), we know that an
integral quadratic form is $k$-isotropic over $\Q$ if and only if it
is $k$-isotropic over $\Q_p$ for all primes $p$ and over $\mathbb{R}$.
Applying a theorem of Poonen and Stoll \cite{PS} we deduce the
following result.

\begin{theorem}
\label{thm:glob}
The probability $\rho_{\rm glob}(k,n)$ that a random integral
quadratic form in $n$ variables is $k$-isotropic exists and is given
by
\[ \rho_{\rm glob}(k,n) = \rho_\infty(k,n) \prod_{p} \rho_p(k,n) \]
where the product is over all primes $p$, and the local contributions
are the probabilities of $k$-isotropy over $\R$ and over $\Q_p$,
defined as in~\eqref{def:global} but with $\Q$ replaced by $\R$ or $\Q_p$
as appropriate.
\end{theorem}

We fix a prime number $p$. The probability $\rho_p(k,n)$ may also be
interpreted as the probability that a random $p$-adic integral
quadratic form in $n$ variables is $k$-isotropic over $\Q_p$. Here, by
a random $p$-adic integral quadratic form, we mean a quadratic form
with coefficients in $\Z_p$ where the coefficients are chosen
independently at random according to Haar measure. Choosing the
coefficient $a_{ij} \in \Z_p$ with respect to Haar measure means that
each congruence class modulo $p$ is equally likely, and inductively
for $n >1$, the classes
\[a,\: a+p^{n-1},\: a+2p^{n-1},\: \ldots,\: a+(p-1)p^{n-1} \mod p^n\]
are equally likely where $0\myleq a \myleq p^{n-1}-1$ is the reduction
of $a_{ij}$ modulo $p^{n-1}$.

We now state our main theorem. 
It extends \cite[Theorem 1.2]{BCFJK} which treated the case $k=1$.

\begin{theorem}
\label{thm:main}
The probability $\rho_p(k,n)$ that a random $p$-adic integral quadratic
form in $n$ variables is $k$-isotropic over $\Q_p$ is given by
\begin{equation*}
  \rho_p(k,n) = \left\{ \begin{array}{ll} \vspace{1ex}
  0 & \text{ if } n \myleq 2k-1; \\ \vspace{1ex}
  \frac{1}{4}\cdot(p^k+1)\cdot\left(\frac{p^{k+2}-1}{(p+1)(p^{2k+1}-1)}
  + \prod_{i=1}^{k}\left(\frac{p^{2i-1}-1}{p^{2i}-1}\right)\right)
  & \text{ if } n = 2k; \\ \vspace{1ex}
 \frac{1}{2}+ \frac{1}{2}\cdot(p^{k+1}+1)\cdot\prod_{i=1}^{k+1}
  \left(\frac{p^{2i-1}-1}{p^{2i}-1}\right) & \text{ if }
  n = 2k+1; \\ \vspace{1ex}
  1-\frac{1}{4}\cdot(p^{k+1}+1)\cdot\left(\frac{p^{k+3}-1}{(p+1)(p^{2k+3}-1)}
  - \prod_{i=1}^{k+1}\left(\frac{p^{2i-1}-1}{p^{2i}-1}\right)\right)
  & \text{ if } n = 2 k + 2; \\
  1 &  \text{ if } n \mygeq 2k + 3. \end{array} \right.
\end{equation*}
\end{theorem}

Combining Theorems~\ref{thm:glob} and~\ref{thm:main} we deduce the following.
\begin{corollary}
\label{cor:prod}
We have $\rho_{\rm glob}(k,n) = 0$ for all $n \myleq 2k + 1$,
\[ \hspace{-1em} \rho_{\rm glob}(k,2k+2) = \rho_\infty(k,2k+2) \cdot
  \prod_p \left( 1 - \frac{p^{k+1}+1}{4}\cdot\left( \frac{p^{k+3} -
        1}{ (p + 1)(p^{2k+3} - 1)} - \prod_{r=1}^{k+1} \frac{p^{2r-1}
        - 1}{p^{2r} - 1} \right) \right) \] and
$\rho_{\rm glob}(k,n) = \rho_\infty(k,n)$ for all $n \mygeq 2k + 3$.
\end{corollary}

We note two striking features of the formulae in
Theorem~\ref{thm:main}.  The first is that they are rational functions
in $p$, where the same rational function works for all primes $p$
including $p=2$. The second is that the rational functions are
invariant upon substituting $p \mapsto 1/p$.  Exactly the same two
observations were made in~\cite{BCFG} in connection with roots of
polynomials in one variable. Moreover in that paper the substitution
$p \mapsto 1/p$ also related two auxiliary probabilities appearing in
the recursion, denoted there by $\alpha$ and $\beta$.  We find that an
analogous statement holds in our case; see
Corollary~\ref{cor:p<->1/p}.

We employ two strategies for proving Theorem~\ref{thm:main}.  The first
is a direct generalisation of the method in \cite{BCFJK} (which only
treated the case $k=1$), with the additional idea of splitting off
hyperbolic planes (see Definition~\ref{def:hyp}). 
This leads to recursive formulae that may be used
to compute $\rho_p(k,n)$ for any given $k$ and $n$, and also show that
the answer is always a rational function in $p$.  However it is still
hard to write the answers uniformly in $k$, as we do in the statement
of Theorem~\ref{thm:main}.
 
The second strategy 
is to deduce Theorem~\ref{thm:main} from a theorem of Kovaleva
\cite{K}, who computed the probability that a random $p$-adic integral
quadratic form in $n$ variables belongs to a given $\Q_p$-equivalence
class of quadratic forms. The answers she obtained are not rational
functions in $p$, do not exhibit the $p \leftrightarrow 1/p$
symmetries, and do not explicitly cover the case $p = 2$, where in any 
case it makes a difference whether we consider random quadratic forms 
or random symmetric matrices. 
However her work leads to a proof of Theorem~\ref{thm:main} when $p$ is odd. 
Since we already know from the first strategy that the answer is a rational
function in~$p$ 
it follows that the theorem is also true when $p=2$.

Both strategies work by dividing into cases according to the
$\F_p$-equivalence class of the quadratic form reduced mod $p$, and
from this obtaining recursive formulae for the probabilities.  One
difference, not already noted above, is that in the second strategy
the quadratic form is diagonalised, whereas in the first we split off
hyperbolic planes, and so allow $2 \times 2$ blocks on the diagonal.

In Section~\ref{sec:background} we review some background on quadratic
forms. In Section~\ref{sec:global} we discuss the global applications
of our work, and in particular explain how Theorem~\ref{thm:glob} and
Corollary~\ref{cor:prod} follow from Theorem~\ref{thm:main}.  In
Section~\ref{sec:count} we prove some results on counting quadratic
forms over finite fields, in preparation for our first strategy for
proving Theorem~\ref{thm:main}. The two strategies are explained in
Sections~\ref{sec:method1} and~\ref{sec:classification} respectively.
Finally in Appendix~\ref{app} we adapt the methods of Kovaleva to
solve the recurrence relations in our first method directly.

This work originated as a summer project carried out by the first author 
and supervised by the second author. We thank the Trinity Summer Studentship 
Scheme and the Research in the CMS Programme for their support.

\section{Background on quadratic forms}
\label{sec:background}
We collect together some standard definitions and results on quadratic
forms.  See Cassels~\cite{Cassels-RQF} for further details.
We write $\F$ for a general field, and $V$ for a finite dimensional
vector space over $\F$.

\begin{definition}
  A \textit{quadratic form} of dimension $n$ over $\F$ is
  a polynomial
\begin{equation}\label{Q2}
    Q(x_1,x_2,\ldots,x_n) = \sum_{1\myleq i \myleq j \myleq n}a_{ij}x_ix_j,
\end{equation}
where the coefficients $a_{ij}$ for $1\myleq i\myleq j\myleq n$ belong
to $\F$. We may also consider $Q$ as a function $V \to \F$ where
$V = \F^n$, and refer to the pair $(V,Q)$ as a {\em quadratic space}.
The corresponding symmetric bilinear form $\phi : V \times V \to \F$
is given by
\begin{equation*}
    \phi(x,y) = Q(x+y)-Q(x)-Q(y),
\end{equation*} where $x=(x_1,x_2,\ldots,x_n)$ and $y=(y_1,y_2,\ldots,y_n)$.
\end{definition}

We refer to properties of a quadratic space $(V,Q)$ and properties of
$V$ or $Q$ interchangeably.

\begin{definition}
  Let $(V,Q)$ be a quadratic space and let $\phi$ be its associated
  symmetric bilinear form.  The \textit{radical} of $(V,Q)$, when not
  over a field of characteristic $2$, is the vector space consisting
  of vectors $x \in V$ such that $\phi(x,y)=0$ for all $y \in V$. In
  characteristic $2$, we further require that $Q(x)=0$.  A quadratic
  space is \textit{regular} if its radical is zero-dimensional and
  \textit{singular} otherwise.  The \textit{rank} of the quadratic
  form $Q$ is $n-r$, where $r$ is the dimension of the radical.
\end{definition}

\begin{definition}
  Quadratic spaces $(V_1,Q_1)$ and $(V_2,Q_2)$ over the same field
  $\F$ are \textit{isometric} if there is a linear isomorphism
  $T: V_1 \rightarrow V_2$ such that $Q_1(x) = Q_2(Tx)$ for all
  $x \in V_1$. In this situation, the forms $Q_1$ and $Q_2$ are said
  to be \textit{equivalent} over $\F$.  In other words, quadratic
  forms over $\F$ are equivalent if they are related by a linear
  substitution given by a matrix $P \in \GL_n(\F)$.  This defines an
  equivalence relation on the set of quadratic forms with coefficients
  in $\F$. More generally, if $R \subset \F$ is a subring, then we say
  that quadratic forms are {\em equivalent over $R$} (or {\em
    $R$-equivalent}) if they are related by a matrix $P \in \GL_n(R)$.
\end{definition}

The next two definitions are closely related. The first naturally
extends the definitions we already gave in the introduction in the
case $\F = \Q$.

\begin{definition}
  Let $(V,Q)$ be a quadratic space. A non-zero vector $x \in V$ is
  called \textit{isotropic} if $Q(x) = 0$. A quadratic space $(V,Q)$
  is {\em isotropic} if $V$ contains an isotropic vector, and
  \textit{totally isotropic} if all its non-zero vectors are
  isotropic. If $V$ has a subspace $V_0$ of dimension $k$ such that
  the quadratic space $(V_0,Q)$ is totally isotropic, then we say that
  the quadratic space $(V,Q)$ is {\em $k$-isotropic}.  In particular,
  a quadratic space is $1$-isotropic if and only if it is isotropic.
\end{definition}

\begin{definition} \label{def:hyp}
  A \textit{hyperbolic plane} is a quadratic space $(V,Q)$ of
  dimension $2$ where $Q$ is equivalent over $\F$ to the form
  $q(x_1,x_2) = x_1x_2$.
\end{definition}

\begin{lemma}\label{isohyp}
  A regular quadratic space $(V,Q)$ is isotropic if and only if $V$
  has a subspace $V_0$ such that the quadratic space $(V_0,Q)$ is a
  hyperbolic plane.
\end{lemma}

\begin{proof}
See \cite[p. 15]{Cassels-RQF}.
\end{proof}

We now introduce some theorems that will be useful for finding
isotropic spaces.

\begin{theorem}[Witt's Cancellation Theorem]\label{Witt}
  Let $(V,Q)$ be a quadratic space. Let $V_1$, $V_2$ be subspaces of
  $V$. Denote by $V_1^\perp$ and $V_2^\perp$ the orthogonal
  complements of $V_1$ resp. $V_2$ in $V$. If $(V_1,Q)$ and $(V_2,Q)$
  are regular and isometric, then $(V_1^\perp,Q)$ and $(V_2^\perp,Q)$
  are also isometric.
\end{theorem}

\begin{proof}
  See \cite[pp. 89--92]{G} for quadratic forms over a field of
  characteristic not $2$, and \cite[p. 118]{G} for the case of
  characteristic $2$.
\end{proof}

Let $p$ be a prime. We write $\F_p$ for the finite field with $p$
elements, $\Q_p$ for the field of $p$-adic numbers, and $\Z_p$ for the
ring of $p$-adic integers. The following argument is one we will
revisit in the proof of Lemma~\ref{deltarelations}.
\begin{lemma}\label{henselcor}
  Let $Q$ be a quadratic form with coefficients in $\Z_p$ that
  reduces over $\F_p$ to a form that is both $k$-isotropic and regular.
  Then $Q$ is $k$-isotropic (and regular) over $\Q_p$.
\end{lemma}

\begin{proof}
  The case $k=1$ is a consequence of Hensel's lemma. In fact if we go
  via Lemma~\ref{isohyp} then we only need Hensel's lemma for a
  quadratic polynomial in one variable. For general $k > 1$ we proceed
  by induction on $k$. Once we know that $Q$ is isotropic, we may
  assume via a $\Z_p$-equivalence that
  \[ Q(x_1, \ldots, x_n) = x_1 x_2 + Q'(x_3, \ldots, x_n) \] for some
  quadratic form $Q'$ with coefficients in $\Z_p$. The reduction of
  $Q'$ mod $p$ is then $(k-1)$-isotropic (and regular) over $\F_p$ by
  Theorem~\ref{Witt}, and the induction hypothesis applies.
\end{proof}

\begin{theorem}\label{CW+M}
  \begin{enumerate}
  \item A quadratic form over $\F_p$ in $3$ or more variables
    is always isotropic.
  \item A quadratic form over $\Q_p$ in $5$ or more variables
    is always isotropic.
  \end{enumerate}
\end{theorem}
\begin{proof}
  (i) This is a consequence of the Chevalley-Warning theorem. See
  for example~\cite[p. 5]{S}. \\
  (ii) This is Meyer's theorem. See for example
  \cite[p. 41]{Cassels-RQF}.
\end{proof}

We make the following definition concerning quadratic forms over $\F_p$.
\begin{definition}
  \label{def:lmn}
  A quadratic form over $\F_p$ belongs to the class $[l,m,n]$ if it is
  equivalent to a form
\begin{equation}\label{modpform}
    Q(x_1,x_2,\ldots,x_n) = \sum_{i=1}^l x_{r+2i-1}x_{r+2i} + f(x_{r+2l+1},\ldots,x_n)
\end{equation}
where $f$ is a regular anisotropic form of dimension $m$. Note that
$l$ is the number of hyperbolic planes in the orthogonal
decomposition, $m \in \{0,1,2\}$ by Theorem~\ref{CW+M}(i), $n$ is the
dimension of the form, and $r = n - 2l - m$ is the dimension of the
radical.
\end{definition}

Repeated application of Lemma~\ref{isohyp} shows that every quadratic
form over $\F_p$ belongs to the class $[l,m,n]$ for some $l, m, n$,
and Theorem~\ref{Witt} shows that $l, m, n$ are uniquely determined by
the quadratic form.

\section{Global densities}
\label{sec:global}

In this section we explain how Theorem~\ref{thm:glob} and
Corollary~\ref{cor:prod} follow from Theorem~\ref{thm:main}.  First we
read off from Theorem~\ref{thm:main} the asymptotic behaviour of
$\rho_p(k,n)$ as $p \to \infty$.

\begin{corollary}\label{asymptotic}
Let $k \mygeq 1$. As $p \to \infty$, we have the following
approximations.
\begin{equation*} \rho_p(k,n) 
    = \left\{ \begin{array}{ll} \vspace{1ex}
     \frac{1}{2} + O(\frac{1}{p}) & \text{ if } n=2k;\\ \vspace{1ex}
     1 - \frac{1}{2p} + O(\frac{1}{p^2}) & \text{ if } n=2k+1;\\ 
     1 - \frac{1}{4p^3} + O(\frac{1}{p^4}) & \text{ if } n=2k+2.
\end{array} \right.
\end{equation*}
\end{corollary}

\begin{proof}
In the case $n = 2k+2$ we consider the Taylor series expansions
\begin{equation*}
  \frac{p^{k+3}-1}{(p+1)(p^{2k+3}-1)} =\frac{1}{p^{k+1}}
  \left(1 - \frac{1}{p} + \frac{1}{p^2} - \frac{1}{p^3}
  + O\left(\frac{1}{p^4}\right)\right),
\end{equation*}
and
\begin{equation*}
  \prod_{i=1}^{k+1} \frac{p^{2i-1}-1}{p^{2i}-1}
  = \frac{1}{p^{k+1}}\left(1 - \frac{1}{p} + \frac{1}{p^2} - \frac{2}{p^3}
  + O\left(\frac{1}{p^4}\right)\right),
\end{equation*}
using the big $O$ notation. Substituting these expansions into the
formula for $\rho_p(k,n)$ in Theorem~\ref{thm:main} gives the
approximation for $\rho_p(k,n)$ as claimed.  The other cases are
similar but easier. \end{proof}

Fix values of $k$ and $n$.  We write $U_\infty$ for the set of
quadratic forms in $\mathbb{R}[x_1,\ldots,x_n]$ that are {\em not}
$k$-isotropic over $\mathbb{R}$.  Likewise we write $U_p$ for the set
of quadratic forms in $\Z_p[x_1,\ldots,x_n]$ that are {\em not}
$k$-isotropic over $\Q_p$.  Let $\mu_\infty$ denote the standard
Lebesgue measure on $\mathbb{R}^d$, and let $\mu_p$ denote the Haar
measure on $\Z_p^d$ normalised to have total volume $1$.

\begin{lemma}
\label{lem:prod}
Let $1 \myleq k \myleq n$ and $d = \binom{n+1}{2}$.  Suppose that the
following condition holds for all sufficiently large primes $p$:
\begin{quotation}
  \noindent
  Every quadratic form in $\Z_p[x_1,\ldots,x_n]$ whose reduction mod
  $p$ has rank at least $n-1$ is $k$-isotropic over $\Q_p$.
\end{quotation}
Then $\rho_{\rm glob}(k,n)$ exists and is given by
\begin{equation}
  \label{eqn:prod} \rho_{\rm glob}(k,n) =
  \frac{\mu_\infty( [-1,1]^d \setminus U_\infty) }{2^d}
  \cdot\prod_p ( 1 - \mu_p(U_p)).
\end{equation}
\end{lemma}
\begin{proof}
  As noted in the introduction, a quadratic form is $k$-isotropic over
  $\Q$ if and only if it is $k$-isotropic over $\Q_p$ for all primes
  $p$ and over $\mathbb{R}$.
  We then apply~\cite[Lemmas 20 and 21]{PS} 
  with $U_\infty$ and $U_p$ as defined above, $S = \emptyset$ and
  $f,g \in \Z[a_{11},a_{12} ,\ldots,a_{nn}]$ two distinct
  $(n -1) \times (n-1)$ minors of the generic $n \times n$ symmetric
  matrix of coefficients.
\end{proof}

In our earlier notation the factors on the right hand side
of~\eqref{eqn:prod} may be written
\begin{equation}
  \label{rho_v}
  \rho_\infty(k,n) = \frac{\mu_\infty( [-1,1]^d \setminus U_\infty) }{2^d}
  \qquad \text{ and } \qquad \rho_p(k,n) = 1 - \mu_p(U_p).
  \end{equation}

\begin{lemma}
  \label{lem:admiss}
  If $n \mygeq 2k + 2$ then the condition in Lemma~\ref{lem:prod} is
  satisfied.
\end{lemma}
\begin{proof} Let $Q \in \Z_p[x_1,\ldots,x_n]$ be a quadratic form
  whose reduction mod $p$ has rank at least $n-1$.  In the terminology
  of Definition~\ref{def:lmn}, the reduction of $Q$ mod $p$ belongs to
  the class $[l,m,n]$ for some $l,m,n$ with $m \in \{0,1,2\}$.  Our
  assumptions then give $2l+m \mygeq n - 1 \mygeq 2 k + 1$. Since $k$
  and $l$ are integers it follows that $l \mygeq k$. Then $Q$ is
  $k$-isotropic over $\Q_p$ by Lemma~\ref{henselcor}.
\end{proof}  

\begin{proofof}{Theorem~\ref{thm:glob}}
Let $\rho_{\rm glob}(k,n)$ be as defined in~\eqref{def:global},
and let $\overline{\rho}_{\rm glob}(k,n)$ be the same quantity with
the limit replaced by lim sup.
  We write $\rho_p(k,n)$ for the probabilities computed in
  Theorem~\ref{thm:main}. A standard argument (see for example
  \cite[Proposition 3.2]{CS}) 
  uses the local conditions at finitely many primes to show that
  \begin{equation}
    \label{finite-prod}
    \overline{\rho}_{\rm glob}(k,n) \myleq \prod_{p < M} \rho_p(k,n).
  \end{equation}

  If $n \myleq 2k+1$ then by Corollary \ref{asymptotic} the right hand
  side of~\eqref{finite-prod} tends to $0$ as $M \to
  \infty$. Therefore $\rho_{\rm glob}(k,n) = 0$ and the equality
  claimed in Theorem~\ref{thm:glob} holds since both sides are zero.

  If $n \mygeq 2k + 2$ then Lemmas~\ref{lem:prod} and~\ref{lem:admiss}
  apply.  Combining~\eqref{eqn:prod} and~\eqref{rho_v} gives
  \[ \rho_{\rm glob}(k,n) =  \rho_\infty(k,n) \prod_p \rho_p(k,n) \]
  as required. 
\end{proofof}

Corollary~\ref{cor:prod} follows immediately from
Theorem~\ref{thm:glob}, Theorem~\ref{thm:main} and the observation in
the last proof that the local product is zero for $n \myleq 2k+1$.

\begin{remark} We do not know an accurate method for computing the
  probabilities $\rho_\infty(k,n)$, but we can estimate them using a
  Monte Carlo simulation.  On this basis we record the following
  numerical values that are likely to be accurate to the number of
  decimal places recorded.
  \[ \begin{array}{c|ccc}
k & \prod_p \rho_p(k,2k+2) & \rho_\infty(k,2k+2) & \rho_{\rm glob}(k,2k+2) \\
       \hline
1 & 0.98743625 & 0.9823 & 0.9699 \\
2 & 0.98229463 & 0.9705 & 0.9533 \\
3 & 0.98007620 & 0.9623 & 0.9431 \\           
4 & 0.97906880 & 0.9561 & 0.9361 \\
5 & 0.97859528 & 0.9512 & 0.9309 
  \end{array} \]
\end{remark}

\begin{remark} In~\cite{BCFJK} (which only treats the case $k=1$) some
  alternatives to the definition~\eqref{def:global} were
  considered. The global densities so defined may still be computed as
  a product over all places, and the local contributions at the finite
  places are the same as before.  However the local contributions at
  infinity can change, and for one natural choice of distribution
  these were computed exactly. It is possible that something similar
  could be done for $k > 1$, but we did not pursue this.
\end{remark}
  
\section{Counting quadratic forms over $\F_{\lowercase{p}}$}
\label{sec:count}

In this section we prove some formulae counting quadratic forms over
$\F_p$.  We consider quadratic forms over $\F_p$ according to their
class $[l,m,n]$ as defined in Definition~\ref{def:lmn}.
\begin{definition}
  \label{pidefs}
  Consider a quadratic form
  \[Q(x_1,x_2,\ldots,x_n) = \sum_{1\myleq i \myleq j \myleq n}a_{ij}x_ix_j\]
  over $\F_p$ where the coefficients $a_{ij}$ are
  chosen independently at random according to counting measure.
  
  Let $\pi_0(l,m,n)$ be the probability that $Q$ belongs to the class
  $[l,m,n]$.

  Let $\pi_1(l,m,n)$ be the probability that $Q$ belongs to the class
  $[l,m,n]$ given that $a_{11} \not= 0$.

  Let $\pi_2(l,m,n)$ be the probability that $Q$ belongs to the class
  $[l,m,n]$ given that $a_{11}x_1^2+a_{12}x_1x_2+a_{22}x_2^2$ is a
  regular anisotropic form.
\end{definition}

Note that Theorem~\ref{CW+M}(i) implies that $\pi_i(l,m,n)=0$ if
$m \mygeq 3$.

The group $\GL_n(\F_p)$ acts on the set of $n$-dimensional quadratic
forms over $\F_p$ by linear substitutions.  The class $[l,m,n]$ is the
union of either one or two orbits, and so its size may be computed as
a sum of orbit sizes.  To begin with we only consider forms that are
regular.  In these cases the orbit sizes can be computed using the
orbit-stabiliser theorem and the following theorem.

\begin{lemma}\label{snk}
  Let $Q_0$ be a quadratic form over $\F_p$ belonging to the class
  $[l,m,n]$. Suppose that $Q_0$ is regular, equivalently $n = 2l+m$.
  Then the stabiliser in $\GL_{n}(\F_p)$ of $Q_0$ is an orthogonal
  group of order $S(m,n)$ where
\begin{equation*}
\begin{aligned}
  S(0,2k) &= 2p^{k(k-1)}(p^k-1)\prod_{i=1}^{k-1}(p^{2i}-1);\\
  S(1,2k+1) &= \left\{ 
  \begin{array}{ll}
  2p^{k^2}\prod_{i=1}^{k}(p^{2i}-1) & \text{ if } p\neq 2;\\
  p^{k^2}\prod_{i=1}^{k}(p^{2i}-1) & \text{ if } p=2;\\
  \end{array}
  \right.\\
  S(2,2k) &= 2p^{k(k-1)}(p^k+1)\prod_{i=1}^{k-1}(p^{2i}-1).
\end{aligned}
\end{equation*}
\end{lemma}
\begin{proof}
  See \cite[pp. 81--82]{G} for $p$ an odd prime,
  and \cite[pp. 147--150]{G} for the case $p=2$.
\end{proof}

To find the orbit size when the radical has dimension $r= n-2l-m$, we
multiply the orbit size of the regular part under the action of
$\GL_{n-r}(\F_p)$ by the number
\begin{equation*}
   N(r,n) = \prod_{i=0}^{r-1} \frac{p^n-p^i}{p^r-p^i}
\end{equation*}
of $r$-dimensional subspaces of $\F_p^n$. 
The orbit size $O(l,m,n)$ of a form in $[l,m,n]$ under the action of
$\GL_n(\F_p)$ is therefore given by
\begin{equation*}
    O(l,m,n)= N(n-2l-m,n)\cdot\frac{|\GL_{2l+m}(\F_p)|}{S(m,2l+m)}.
\end{equation*}

If $m=1$ and $p$ is odd, there are two orbits belonging to $[l,m,n]$;
other values of $m$ give a unique orbit (see \cite[p. 79]{G}).  Hence,
using that $|\GL_n(\F_p)| = \prod_{i=0}^{n-1}(p^n-p^i)$, and dividing
by the total number of quadratic forms of dimension $n$, we obtain the
values of $\pi_0(l,m,n)$ recorded in the next lemma.  Note that the
only form in the class $[0,0,n]$ is the form where all the
coefficients are zero.

\begin{lemma}\label{pi0} For $l+m>0$ and $n = 2l + m + r$ we have
  \[
    \pi_0(l,m,n) = \left\{ \begin{array}{ll} \vspace{1ex}
 \frac{1}{p^{n(n+1)/2}}\cdot\prod_{i=0}^{r-1}\frac{p^{n}-p^i}{p^{r}-p^i}\cdot
 \frac{\prod_{i=0}^{2l-1}(p^{2l}-p^i)}{2p^{l(l-1)}(p^{l}-1)\prod_{i=1}^{l-1}(p^{2i}-1)}
 & \text{ if } m = 0; \\ \vspace{1ex}
 \frac{1}{p^{n(n+1)/2}}\cdot\prod_{i=0}^{r-1}\frac{p^{n}-p^i}{p^{r}-p^i}
 \cdot\frac{\prod_{i=0}^{2l}(p^{2l+1}-p^i)}{p^{l^2}\prod_{i=1}^{l}(p^{2i}-1)}
  & \text{ if } m = 1; \\
  \frac{1}{p^{n(n+1)/2}}\cdot\prod_{i=0}^{r-1}\frac{p^{n}-p^i}{p^{r}-p^i}
  \cdot \frac{\prod_{i=0}^{2l+1}(p^{2l+2}-p^i)}{2p^{l(l+1)}(p^{l+1}+1)
  \prod_{i=1}^{l}(p^{2i}-1)} & \text{ if } m = 2. \end{array} \right.              
\]
Moreover, $\pi_0(0,0,n) = 1/p^{n(n+1)/2}$.
\end{lemma}

Next we compute the probabilities $\pi_1(l,m,n)$ in terms of the
probabilities $\pi_0(l,m,n)$.

\begin{lemma} We have
 \begin{equation*}
  \pi_1(l,m,n) = \left\{ \begin{array}{ll}
  \pi_0(l-1,1,n-1)/2 & \text{ if } m=0 \text{ and } l \mygeq 1; \\
  \pi_0(l,0,n-1)+\pi_0(l-1,2,n-1) & \text{ if } m=1 \text{ and } l \mygeq 1; \\
  \pi_0(l,1,n-1)/2 & \text{ if } m=2. \\
\end{array} \right.                               
\end{equation*}
Moreover, $\pi_1(0,0,n) = 0$ and $\pi_1(0,1,n) = 1/p^{n(n-1)/2}$.
\end{lemma}                         
  
\begin{proof}
  We first suppose that $p$ is an odd prime.  Let $Q$ be a quadratic
  form in $n$ variables over $\F_p$ with first coefficient
  $a_{11} \not= 0$.  We must compute the probability $\pi_1(l,m,n)$
  that $Q$ belongs to the class $[l,m,n]$.  By a linear substitution
  to eliminate the cross-terms containing $x_1$, we may assume that
  $Q(x_1,x_2,\ldots,x_n) = a_{11}x_1^2 + F(x_2,x_3,\ldots,x_n)$ for
  some $F \in \F_p[x_2,\ldots,x_n]$.  The class of $Q$ is determined
  by the class of $F$ and the value of $a_{11}$. Since the
  coefficients of $F$, like those of $Q$, are randomised according to
  counting measure, we can compute $\pi_1(l,m,n)$ in terms of the
  $\pi_0(l',m',n-1)$ for suitable $l'$ and $m'$.  More precisely, we
  note that if $F$ belongs to the class $[l,0,n-1]$ or $[l-1,2,n-1]$
  then $Q$ belongs to the class $[l,1,n]$, whereas if $F$ belongs to
  the class $[l-1,1,n-1]$ then it is equally likely that $Q$ belongs
  to the class $[l,0,n]$ or $[l-1,2,n]$. The stated formulae follow.

  To prove the lemma when $p=2$ we outline an alternative method for
  computing $\pi_1(l,m,n)$ that gives the answer as a rational
  function in $p$.  In this alternative method we compute
  $\pi_1(l,m,n)$ by finding the probability that a form in $[l,m,n]$
  satisfies $a_{11} \not= 0$, and then multiply by
  $\pi_0(l,m,n) \cdot \frac{p}{p-1}$ according to Bayes' formula. The
  second factor comes from the fact that $a_{11} \not=0$ with
  probability $\frac{p-1}{p}$.  Since $a_{11} = Q(1,0, \ldots,0)$ and
  $\GL_n(\F_p)$ acts transitively on $\F_p^n \setminus \{0\}$ it
  suffices to show that \[N(Q) = \#\{ x \in \F_p^n \mid Q(x) = 0\}\]
  is a polynomial in $p$, where the polynomial depends only on
  $l,m,n$. We prove this claim by induction on $l$, noting that if
  $Q(x_1, \ldots,x_n) = x_1 x_2 + Q'(x_3, \ldots, x_n)$ then
  $N(Q) = (2p-1)N(Q') + (p-1)(p^{n-2} - N(Q'))$, whereas if $l = 0$
  then $N(Q) = p^{n-m}$.
\end{proof}

To determine the values of $\pi_2(l,m,n)$, we use a method similar to
the one we used for calculating $\pi_1(l,m,n)$ for $p$ an odd
prime. However, this proof also includes the case $p=2$.

\begin{lemma} We have
  \[ \pi_2(l,m,n) = \left\{ \begin{array}{ll}
    \pi_0(l-2,2,n-2) &\text{ if } m = 0 \text{ and } l\mygeq 2;\\
    \pi_0(l-1,1,n-2) &\text{ if } m = 1 \text{ and } l\mygeq 1;\\
    \pi_0(l,0,n-2)  &\text{ if } m = 2.
  \end{array} \right. \]
Moreover,
$\pi_2(0,0,n) = \pi_2(0,1,n) = \pi_2(1,0,n) = 0$.
\end{lemma}
\begin{proof}
  It suffices to consider
  $Q(x_1,x_2,\ldots,x_n) = f(x_1,x_2) + F(x_3,\ldots,x_n)$ for
  $f \in \F_p[x_1,x_2]$ regular anisotropic and
  $F \in \F_p[x_3,\ldots,x_n]$.  The class of $F$ then determines the
  class of $Q$, and this gives the formulae as stated.
\end{proof}

\section{First method: Reduction modulo $\lowercase{p}$ and recursion}
\label{sec:method1}

In this section we give our first method for computing the probability
$\rho_p(k,n)$ that a random $p$-adic integral quadratic form in $n$
variables is $k$-isotropic.

\begin{definition}\label{deltadefs}
  Let $Q$ be a random $p$-adic integral quadratic form in $n$
  variables.
 
  Let $\delta_0(k;l,m,n)$ be the probability that $Q$ is $k$-isotropic
  given that its reduction mod $p$ belongs to the class $[l,m,n]$.

  Let $\delta_1(k;l,m,n)$ be the probability that $Q$ is $k$-isotropic
  given that its reduction mod $p$ belongs to the class $[l,m,n]$, the
  coefficients $a_{11},a_{12}, \ldots, a_{1n}$ are all divisible by
  $p$, but $p^2$ does not divide $a_{11}$.
  
  Let $\delta_2(k;l,m,n)$ be the probability that $Q$ is $k$-isotropic
  given that its reduction mod $p$ belongs to the class $[l,m,n]$, the
  coefficients $a_{11},a_{12}, \ldots, a_{1n}$ and
  $a_{22}, a_{23}, \ldots, a_{2n}$ are all divisible by $p$, but the
  reduction of $\frac{1}{p}(a_{11}x_1^2+a_{12}x_1x_2+a_{22}x_2^2)$ mod
  $p$ is a regular anisotropic form.
\end{definition}

By definition $\delta_0(k;0,0,n)$ is the probability of $k$-isotropy
given that $Q$ vanishes mod $p$. This is the same as $\rho_p(k,n)$.
Our next two results establish recursive relations for computing the
$\delta_i(k;l,m,n)$.

\begin{lemma}\label{deltarelations}
 For $i \in \{0,1,2\}$ we have 
\begin{equation*}
  \delta_i(k;l,m,n) =
  \left\{ \begin{array}{ll} \delta_i(k-l;0,m,n-2l) &
  \text{ if } k > l; \\
  1 & \text{ if } k \myleq l. \end{array} \right. 
\end{equation*}
\end{lemma}

\begin{proof}
  A quadratic form whose reduction modulo $p$ belongs to the class
  $[l,m,n]$ is equivalent over $\Z_p$ to a form which satisfies
\begin{equation}\label{2ndmodpform}
    Q(x_1,\ldots,x_n) \equiv \sum_{i=1}^l x_{r+2i-1}x_{r+2i} + f(x_{r+2l+1},\ldots,x_{n}) \mod p,
\end{equation}
for $f$ a regular anisotropic form over $\F_p$ of dimension
$m \in \{0,1,2\}$. We claim that $Q$ is equivalent over $\Z_p$ to a form
\begin{equation}\label{2ndpadicform}
  Q'(x_1,\ldots,x_n) = \sum_{i=1}^l x_{r+2i-1}x_{r+2i}
  + Q''(x_1,\ldots,x_r,x_{r+2l+1},\ldots,x_{n}),
\end{equation}
where $Q''(x_1,\ldots,x_r,x_{r+2l+1},\ldots,x_{n})
\equiv f(x_{r+2l+1},\ldots,x_{n}) \mod p$.
If $k \myleq l$ it follows immediately that $Q$ is $k$-isotropic.  If
$k > l$ then by Witt's Cancellation Theorem, the form $Q$ is
$k$-isotropic if and only if the form $Q''$ is $(k-l)$-isotropic.  So
it only remains to prove the claim, and at the same time convince
ourselves that $Q''$ is suitably randomised.

For simplicity let us suppose that $l=1$. Then the $\Z_p$-equivalence
taking~\eqref{2ndmodpform} to~\eqref{2ndpadicform} is built out of two
sorts of transformations. First we let $\GL_2(\Z_p)$ act on the
variables $x_{r+1}$ and $x_{r+2}$ by linear substitution. (As in the
proof of Lemma~\ref{henselcor} the
required substitution exists by Hensel's lemma.)  Then we make
substitutions for $x_{r+1}$ and $x_{r+2}$ where we add to each $p$
times a linear combination of the other variables.  Thus
$Q''(x_1, \ldots, x_r,x_{r+3}, \ldots, x_n)$ is obtained from
$Q(x_1, \ldots, x_r,0,0,x_{r+3}, \ldots, x_n)$ by adding a form which
vanishes mod $p^2$. Since the extra conditions on $Q$ in the
definition of the $\delta_i$ for $i=1,2$ are conditions on the
coefficients mod $p^2$, these are not affected by this change.
\end{proof}

\begin{lemma}\label{recrel}
For $i,j \in \{0,1,2\}$ and $n\mygeq i+j$ we have
\begin{equation*}
\begin{aligned}
  \delta_i(k;0,j,n) &= \sum_{l \mygeq 0}
  \sum_{m=0}^{2}\pi_i(l,m,n-j)\delta_j(k;l,m,n). 
\end{aligned}
\end{equation*}
Moreover, if $n = i + j$ then \[\delta_i(k;0,j,n) = \left\{
    \begin{array}{ll} 1 & \text{ if } k = 0; \\
    0 & \text{ if } k \mygeq 1. \end{array} \right. \]
\end{lemma}

The condition $n\mygeq i+j$ ensures that $\pi_i(l,m,n-j)$ is defined.
It can only be non-zero if $2 l + m \myleq n-j$, in which case
$\delta_j(k;l,m,n)$ is defined. In particular the sum over $l$ is
finite.

\begin{proof}
  Let $Q$ be a $p$-adic integral quadratic form of dimension $n$
  whose reduction mod $p$ belongs to the class $[0,j,n]$.
  By an equivalence over $\Z_p$ we may suppose that the reduction
  of $Q$ mod $p$ is an anisotropic form in the
  last $j \in \{0,1,2\}$ variables. We multiply each of the last
  $j$ variables by $p$, and divide through by $p$ to give a
  new $p$-adic integral quadratic form, whose reduction mod $p$
  involves only the first $n-j$ variables. If $i=1$ or $2$ then
  the additional conditions in Definition~\ref{deltadefs}
  give the additional conditions in
  Definition~\ref{pidefs}. The reduction mod $p$ now has
  class $[l,m,n]$ with probability $\pi_i(l,m,n-j)$, and in
  this case the form is $k$-isotropic with probability
  $\delta_j(k;l,m,n)$. In checking this last statement, notice that
  the extra conditions in Definition~\ref{deltadefs} when $j=1$ or $2$
  are satisfied relative to the last $j$ variables rather than
  the first $j$ variables. This change clearly does not matter.
  Summing over all possibilities for $l$ and $m$ gives the result.

  For the final part we show that if $n = i + j$ then the forms
  considered in the definition of $\delta_i(k;0,j,n)$ are anisotropic
  over $\Q_p$. For example, if $i=j=2$ and $n=4$ then the reduction of
  $Q(x_1,\ldots,x_4)$ mod~$p$ is an anisotropic form in $x_3$ and
  $x_4$, and the reduction of $\frac{1}{p} Q(x_1,x_2,px_3,px_4)$ mod
  $p$ is an anisotropic form in $x_1$ and $x_2$. Supposing that
  $Q(a_1, \ldots, a_4)=0$ for some $a_1,\ldots,a_4 \in \Z_p$ not all
  divisible by $p$ these conditions quickly lead to a contradiction.
  The other cases are similar.
\end{proof}

\begin{proposition}
  \label{cansolve}
  The relations in Lemmas~\ref{deltarelations} and~\ref{recrel} are
  sufficient to determine all the $\delta_i(k;l,m,n)$ and to show that
  they are rational functions in $p$.  The same is therefore true of
  $\rho_p(k,n) = \delta_0(k;0,0,n)$.
\end{proposition}

\begin{proof}
  Combining the two lemmas shows that
  \[ \delta_i(k;0,j,n) = \frac{1}{p^{\binom{n+1-i-j}{2}}}
    \delta_j(k;0,i,n) + \ldots \] where the terms omitted involve
  either a smaller value of $n$ or a larger value of $i+j$. Assuming
  all such previous values have been computed, we can uniquely solve
  for $\delta_i(k;0,j,n)$ and $\delta_j(k;0,i,n)$ provided that
  $n > i + j$. It is clear from Definition~\ref{deltadefs} that we
  must have $n \mygeq i+j$ and the remaining case where $n=i+j$ is
  covered by the last part of Lemma~\ref{recrel}.  Finally we use
  Lemma~\ref{deltarelations} to compute the $\delta_i(k;l,m,n)$ with
  $l > 0$.

  Since we saw in Section~\ref{sec:count} that the $\pi_i(l,m,n)$ are
  rational functions in $p$, it follows that the $\delta_i(k;l,m,n)$
  are also rational functions in $p$.
\end{proof}  

Proposition~\ref{cansolve} together with the results of the next
section are all we shall need for the proof of Theorem~\ref{thm:main}.
It is nonetheless still interesting to find explicit closed formulae
for the $\delta_i(k;l,m,n)$. We do this now, leaving some of the
details to Appendix~\ref{app}.

\begin{definition}
\label{def:phipsi}
For $i,j \in \{0,1,2\}$ and $n \mygeq i + j$ we define
\begin{align*}
  \phi(i,j,n) &= \big( (j - 1) p^d + (i - 1) \big) \cdot
            \prod_{r=1}^d \frac{p^{2r-1} - 1}{p^{2r} - 1}, \\
  \psi(i,j,n) &= \frac{\big((j - 1) p^d + (i - 1) \big)
  \big( (j - 1) p^{d+2} - (i - 1) \big) - \delta_{i1} p
  + \delta_{j1} p^{2 d + 1}}{ (p + 1)(p^{2d+1} - 1)},
\end{align*}
where $d = \lfloor \frac{n+1-i-j}{2} \rfloor$ and $\delta_{ij}$ is
the Kronecker delta.
\end{definition}

\begin{proposition}\label{prop:solve-recur}
  Let $i,j \in \{0,1,2\}$ and $n \mygeq i+j$. Then
  \begin{equation*}
    \phi(i,j,n) = \sum_{\myell \mygeq 0} \sum_{m=0}^2 \pi_i(\myell,m,n-j)
    \phi(j,m,n-2\myell),
  \end{equation*}
  and if $n$ is {\bf even} then 
  \begin{equation*} 
  \psi(i,j,n) = \sum_{\myell \mygeq 0} \sum_{m=0}^2 \pi_i(\myell,m,n-j)
                  \psi(j,m,n-2\myell).
  \end{equation*}
\end{proposition}
\begin{proof} We prove this in Appendix~\ref{app} by adapting methods
   of Kovaleva \cite{K}. 
\end{proof}

Theorem~\ref{thm:main} is the special case $i=j=0$ of the following
result.

\begin{theorem}
\label{thm:deltas}
For any $i,j \in \{0,1,2\}$ and $n \mygeq i + j$ we have
\[
  \delta_i(k;0,j,n) = \left\{ \begin{array}{ll} \vspace{1ex}
  0 & \text{ if } n \myleq 2k-1; \\ \vspace{1ex}
  \frac{1}{4} \big( -\phi(i,j,n) + \psi(i,j,n) \big)
  & \text{ if } n = 2k; \\  \vspace{1ex}
  \frac{1}{2} \big( 1 - \phi(i,j,n) \big) & \text{ if } n = 2 k + 1;
    \\ \vspace{1ex}
  1 - \frac{1}{4} \big( \phi(i,j,n) + \psi(i,j,n) \big)
               & \text{ if } n = 2k + 2; \\
             1 & \text{ if } n \mygeq 2k+3. \end{array} \right. \]
\end{theorem}
\begin{proof} By Proposition~\ref{prop:solve-recur} these are
  solutions to the recurrence relations in Proposition~\ref{cansolve}. These
  particular linear combinations of $1$, $\phi$ and $\psi$ also
  satisfy the initial conditions, that is, we checked they give the
  correct answers when $n = i + j$.
\end{proof}

As explained in the introduction, the following corollary is
interesting since it generalises a phenomenon studied in \cite{BCFG}.

\begin{corollary}
  \label{cor:p<->1/p}
  The probabilities $\delta_i(k;l,j,n)$ and $\delta_j(k;l,i,n)$ are
  rational functions in $p$ that are exchanged when we replace $p$ by
  $1/p$. In particular $\rho_p(k,n) = \delta_0(k;0,0,n)$ is unchanged
  when we replace $p$ by $1/p$.
\end{corollary}
\begin{proof} By Lemma~\ref{deltarelations} it suffices to prove the
  case $l = 0$. The symmetries claimed then follow from
  Definition~\ref{def:phipsi} and Theorem~\ref{thm:deltas}.
\end{proof}  

\section{Second method: Using a theorem of Kovaleva}
\label{sec:classification}

In this section we deduce Theorem~\ref{thm:main} from a result of
Kovaleva \cite{K}. First we recall the classification of quadratic
forms over $\Q_p$ up to equivalence.

\begin{definition}
  Let $a$, $b \in \Q_p^*$. The \textit{Norm-Residue symbol}, denoted
  $\binom{a,b}{p}$ or more simply as $(a,b)$, is set to be $1$ when
  the form $ax^2+by^2-z^2$ vanishes for some $x,y,z\in \Q_p$ not all
  zero, and $-1$ otherwise.
\end{definition}


\begin{definitionlemma}
  Let $Q \in \Q_p[x_1,\ldots,x_n]$ be a quadratic form of rank $n$
  which is equivalent to a diagonal form
  $Q'(x_1,\ldots,x_n) = \sum_{i=1}^n a_{i}x_i^2$. The
  \textit{Hasse-Minkowski invariant} of the form $Q$ is defined as
  $c(Q) = \prod_{i<j}(a_i,a_j)$. This is independent of the choice of
  diagonal form.
\end{definitionlemma}

\begin{proof}
See \cite[pp. 56--58]{Cassels-RQF}.
\end{proof}


\begin{theorem}
  A quadratic form $Q \in \Q_p[x_1,\ldots,x_n]$ of rank $n$ is uniquely
  determined up to $\Q_p$-equivalence by its determinant $d(Q)
  \in \Q_p^*/(\Q_p^*)^2$ and its Hasse-Minkowski invariant $c(Q) \in \{\pm 1\}$.
\end{theorem}
\begin{proof}
See \cite[p. 61]{Cassels-RQF}.
\end{proof}

The next lemma explains why we only need to consider forms of full
rank over $\Q_p$.

\begin{lemma}\label{fullrank}
  A $p$-adic integral quadratic form, with coefficients chosen
  independently from $\Z_p$ according to Haar measure, is singular
  with probability zero.
\end{lemma}

\begin{proof}
  We write the form as
  $Q(x_1,x_2,\ldots,x_n) = a_{11}x_1^2+ x_1\cdot f(x_2,\ldots,x_n) +
  g(x_2,\ldots,x_n)$ for some linear form $f \in \Z_p[x_2,\ldots,x_n]$
  and quadratic form $g \in \Z_p[x_2,\ldots,x_n]$.

  If $n=1$, the form is singular when $a_{11}$ is zero, which happens
  with probability zero. Inductively, for $n>1$, we can assume the
  form $g$ to be non-singular. For each linear form $f$ and
  non-singular form $g$, there is only one value of $a_{11}$ that
  makes $Q$ singular, corresponding to the determinant of the
  coefficient matrix being zero. This value is attained by $a_{11}$
  with probability zero, hence the form is singular with probability
  zero by induction.
\end{proof}

We now take $p$ an odd prime. The following theorem, due to Kovaleva,
gives for each triple $(n,d,c)$ the probability that a random $p$-adic
integral quadratic form $Q$ in $n$ variables has determinant
$d(Q) = d$ and Hasse-Minkowski invariant $c(Q) = c$.  Since $p$ is an
odd prime, the quotient $\Q_p^*/(\Q_p^*)^2$ has order $4$, with coset
representatives $\{1,u,p,up\}$ where $u$ is a quadratic non-residue
modulo $p$.

\begin{theorem}[Kovaleva] \label{kovaleva} Let $p$ be an odd prime,
  and let $Q\in \Z_p[x_1,x_2,\ldots,x_n]$ be a random $p$-adic
  integral quadratic form in $n$ variables.  Let $\eps$ and $s$ denote
  the Legendre symbols $(\frac{-1}{p})$ resp. $(\frac{d}{p})$ and let
  $u$ be a quadratic non-residue modulo $p$.  Then the probability
  $\mathbb{P}_n(d(Q)=d, c(Q) =c)$ that $Q$ has determinant
  $d \in \Q_p^*/(\Q_p^*)^2$ and Hasse-Minkowski invariant
  $c \in \{\pm 1\}$ is given by
\begin{equation*}
\begin{aligned}
  \mathbb{P}_{2k+1}(d(Q)=d, c(Q) =c) &= \left\{
    \begin{aligned}
      &\frac{1}{4}\cdot\frac{p}{p+1}+ \frac{1}{4}\cdot c\cdot p^{k+1}\cdot\prod_{i=1}^{k+1}\frac{p^{2i-1}-1}{p^{2i}-1} && \text{if} \hspace{6pt}d \in \{ 1,u\};\\
      &\frac{1}{4}\cdot \frac{1}{p+1}+ \frac{1}{4}\cdot c\cdot \eps^{k}\cdot\prod_{i=1}^{k+1}\frac{p^{2i-1}-1}{p^{2i}-1} && \text{if} \hspace{6pt}d \in \{ p,up\};\\
    \end{aligned}
  \right.\\
  \mathbb{P}_{2k}(d(Q)=d, c(Q) =c) &= \left\{
    \begin{aligned}
      &\frac{1}{4}\cdot(p^k+s\eps^{k})\cdot\left(\frac{(p^{k+2}-s\eps^{k})}{(p+1)(p^{2k+1}-1)}+ c\cdot \prod_{i=1}^{k}\frac{p^{2i-1}-1}{p^{2i}-1}\right) && \text{if} \hspace{6pt}d \in \{ 1,u\};\\
      &\frac{1}{4}\cdot\frac{p}{p+1}\cdot\frac{p^{2k}-1}{p^{2k+1}-1} && \text{if} \hspace{6pt}d \in \{ p,up\}. \\
    \end{aligned}
    \right.
\end{aligned}
\end{equation*}
\end{theorem}
\begin{proof}
See \cite[Theorem 1.3]{K}.
\end{proof}

We deduce Theorem~\ref{thm:main} for $p$ odd using the following
lemma.  Recall that we wrote $\rho_p(k,n)$ for the probability that a
random $p$-adic integral quadratic form in $n$ variables is
$k$-isotropic.

\begin{lemma}
  \label{rho-P}
  We have $\rho_p(k,n) = 0$ for $n \myleq 2k-1$ and $\rho_p(k,n) = 1$
  for $n \mygeq 2k+3$. If $p$ is odd then
  \begin{align*}
    \rho_p(k,2k) &= \mathbb{P}_{2k}(d(Q)=(-1)^k, c(Q) = 1); \\
    \rho_p(k,2k+1) &= \sum_{a \in \Q_p^*/(\Q_p^*)^2}
           \mathbb{P}_{2k+1}(d(Q)=(-1)^k a , c(Q) = (-1,a)^k); \\
    \rho_p(k,2k+2) &= 1 - \mathbb{P}_{2k+2}(d(Q)=(-1)^{k-1} , c(Q) = -1).
  \end{align*}
\end{lemma}  
\begin{proof}
  We first note that if $n \myleq 2k-1$ then every $k$-isotropic form
  of dimension $n$ is singular, and so $\rho_p(k,n)=0$ by
  Lemma~\ref{fullrank}.  We now suppose that $n \mygeq 2k$.  By
  Lemma~\ref{isohyp} every regular quadratic form of dimension $n$
  over $\Q_p$ is equivalent to one of the form
  \begin{equation}\label{newhypform}
    Q(x_1,\ldots,x_n) = \sum_{i=1}^l x_{2i-1} x_{2i} + f(x_{2l+1},\ldots,x_{2l+m}),
  \end{equation}  
  where $l$ is the number of hyperbolic planes in the decomposition,
  $f(x_{2l+1},\ldots,x_{2l+m})$ is an anisotropic form over $\Q_p$ of
  rank $m$, and $n = 2l + m$. By Theorem~\ref{CW+M}(ii) we have
  $m \myleq 4$. It follows that if $n \mygeq 2k+3$ then $k \myleq l$
  and so $\rho_p(k,n)=1$.

  We now take $p$ an odd prime. If $n = 2k$ then for the form 
  in~\eqref{newhypform} to be $k$-isotropic we
  need $l=k$ and $m=0$. There is only one such form up to
  $\Q_p$-equivalence. It has determinant $d(Q) = (-1)^k$ and
  Hasse-Minkowski invariant $c(Q) = 1$. This gives the formula for
  $\rho_p(k,2k)$ as stated. If $n = 2k+1$ then for $k$-isotropy we
  need $l=k$ and $m=1$. The anisotropic form in~\eqref{newhypform} is
  $f(x_n) = a x_n^2$ for some $a \in \{1,u,p,up\}$. This gives four
  $\Q_p$-equivalence classes of forms, with invariants
  $d(Q) = (-1)^ka$ and $c(Q) = (-1,a)^k$. Finally we take $n =
  2k+2$. For $Q$ {\em not} to be $k$-isotropic we need $l = k-1$ and
  $m = 4$. The rank $4$ anisotropic form $f$ has determinant
  $d(f) = 1$ and Hasse-Minkowski invariant $c(f) = -1$ (see
  \cite[p. 59]{Cassels-RQF}). It follows that $d(Q) = (-1)^{k-1}$ and
  $c(Q) = -1$, giving the result as stated.
\end{proof}

Theorem~\ref{thm:main} for $p$ odd now follows from
Theorem~\ref{kovaleva} and Lemma~\ref{rho-P}. The interesting thing to
note is that the Legendre symbols $\eps$ and $s$ cancel, giving
answers that are rational functions in $p$. Indeed when $n=2k$ we have
$s \eps^k = 1$. When $n=2k+1$ the contributions for $a \in \{1,u\}$
have $c = 1$ and the contributions for $a \in \{p,up\}$ have
$c = \eps^k$. When $n = 2k+2$ we employ the corresponding formula in
Theorem~\ref{kovaleva} with $k$ replaced by $k+1$ and $s$ replaced by
$\eps^{k-1}$.

We saw in Proposition~\ref{cansolve} that $\rho_p(k,n)$ is given by a
rational function in $p$, where the same rational function works for
all primes $p$ including the prime $p=2$. Since we proved
Theorem~\ref{thm:main} in the last paragraph for infinitely many
primes (in fact for all odd primes), the theorem is therefore true for all
primes.

\appendix

\section{Solving the recurrence relations for the first method}
\label{app}

In this appendix we prove Proposition~\ref{prop:solve-recur}.  This is
not needed for the proof of our main theorems as stated in the
introduction, but is needed to compute all the $\delta_i(k;l,m,n)$
(see Theorem~\ref{thm:deltas}) and hence to see that they satisfy some
interesting symmetries (see Corollary~\ref{cor:p<->1/p}).  The proof
is based on that of Theorem~\ref{kovaleva}, but we could
not see a way to directly cite Kovaleva's work without reworking all
the details. 

The identities we seek to prove are ones between rational functions in
$p$.  So it suffices to prove them for any infinite set of primes.
There is therefore no loss of generality in assuming (as we now do)
that $p$ is odd. This allows us to identify quadratic forms and
symmetric matrices in the usual way.

\begin{definition} Let $\lambda(r,n)$ be the probability that a
  randomly chosen $n \times n$ symmetric matrix over $\F_p$ has rank
  $r$. This is computed in \cite[Corollary 4.2]{K}. 
  In the notation of Section~\ref{sec:count} we have
  \begin{equation} \label{lambdas}
  \lambda(r,n) =  \left\{ \begin{array}{ll} \pi_0(\myell,0,n) +
    \pi_0(\myell-1,2,n) &
   \text{ if } r = 2 \myell; \\ \pi_0(\myell,1,n)
   & \text{ if } r = 2 \myell + 1.  \end{array} \right.
   \end{equation}
\end{definition}

\begin{lemma}
  \label{lem:id1}
  For any $x \in \R$ we have
  \[\sum_{\substack{r=0 \\ n-r \text{ even }}}^n \lambda(r,n) \frac{x-p^r}{p^{n+1}-p^r}
    \,\,+ \sum_{\substack{r=0 \\ n-r \text{ odd }}}^n \lambda(r,n) = \left\{
  \begin{array}{ll} \vspace{1ex} \displaystyle\frac{x-1}{p^{n+1}-1}
    & \text{ if $n$ is even;} \\
    \displaystyle\frac{x}{p^{n+1}} & \text{ if $n$ is odd}.
  \end{array} \right. \]
\end{lemma}
\begin{proof} Since each side is linear in $x$, it suffices to prove
  the identity for just two values of $x$. If $x = p^{n+1}$ then this
  is just the fact that $\sum_{r=0}^n \lambda(r,n) = 1$. If $n = 2k$
  and $x=1$ then the left hand side is
  \[ \sum_{s=1}^k \left( \lambda(2s,n) \frac{1 - p^{2s}}{p^{n+1}-p^{2s}}
    + \lambda(2s-1,n) \right) \]
  whereas if $n=2k+1$ and $x=0$ then the left hand side is
  \[ \sum_{s=1}^k \left( \lambda(2s,n) - \lambda(2s+1,n)
    \frac{p^{2s}}{p^{n+1}-p^{2s}} \right). \]
  It may be checked using~\eqref{lambdas} 
  that in each of these last two sums all the summands are zero.
\end{proof}

\begin{lemma}
  \label{lem:os}
  Let $\pi_n = \prod_{i=1}^n (1 - p^{-i})$. Then for any $m,n \mygeq 0$
  we have
  \[ \sum_{r=0}^{\min(m,n)} \frac{\pi_m \pi_n}{p^{(m-r)(n-r)}
      \pi_r \pi_{m-r} \pi_{n-r}} = 1. \]
\end{lemma}
\begin{proof}
  This is \cite[Corollary 2.5]{K}. The $r$th summand is the
  probability that an $m \times n$ matrix over $\F_p$ has
  rank~$r$. This may be computed by considering the action of
  $\GL_m(\F_p) \times \GL_n(\F_p)$ via $(A,B): X \mapsto A X B^T$ and
  applying the orbit-stabiliser theorem.
\end{proof}

We define 
\begin{equation}
  \label{eqn:AB}
  A(n) = \prod_{i=1}^{ \lfloor \frac{n+1}{2} \rfloor} \frac{p^{2i}-p}{p^{2i} -1},
      \qquad 
  B(n) = \prod_{i=1}^{ \lfloor \frac{n+1}{2} \rfloor} \frac{p^{2i-1}-1}{p^{2i} -1}.
\end{equation}
  
\begin{lemma}
  \label{lem:id2}
  For any $x,y,z \in \R$ we have
  \begin{align*}
\sum_{\substack{r=0 \\ n-r \text{ even }}}^n & \lambda(r,n)B(n-r) (x + y p^{n-r})
\,\,+ \sum_{\substack{r=0 \\ n-r \text{ odd }}}^n \lambda(r,n)B(n-r)(x + z  p^{n+1-r})
  \\ & = \left\{
 \begin{array}{ll} \vspace{1ex} \left( x +  y
 + \left(1 - \displaystyle\frac{1}{p^{n}} \right) z \right) A(n) 
  & \text{ if $n$ is even;} \\
  \left(x + \left(1 - \displaystyle\frac{1}{p^{n+1}} \right)y + z \right) A(n)
  & \text{ if $n$ is odd}.
  \end{array} \right. \end{align*}
\end{lemma}
\begin{proof}
  It suffices to prove this identity for three linearly independent
  choices of $(x,y,z)$.  If $(x,y,z) = (1,-1,0)$ and $n$ is even or
  $(x,y,z) = (1,0,-1)$ and $n$ is odd then the terms with $r = 2s$ and
  $r = 2s+1$ cancel, giving the result in these cases. The proof is
  completed by the next lemma which proves the cases
  $(x,y,z) = (0,1,0)$ and $(x,y,z)=(0,0,1)$.
\end{proof}

\begin{lemma}
  \label{lem:twoid}
  We have
  \[ \sum_{\substack{\myell=0 \\ \myell \text{ even }}}^n \lambda(n-\myell,n)
    B(\myell) p^\myell = \left\{ \begin{array}{ll}
    A(n) 
    & \text{ if $n$ is even; } \\
    \left(1 - \frac{1}{p^{n+1}} \right) A(n) 
     & \text{ if $n$ is odd, }
    \end{array} \right. \]
     and                    
  \[ \sum_{\substack{\myell=0 \\ \myell \text{ odd }}}^n \lambda(n-\myell,n)
    B(\myell) p^{\myell+1} = \left\{ \begin{array}{ll}
    \left(1 - \frac{1}{p^{n}} \right) A(n) 
    & \text{ if $n$ is even; } \\
  A(n) 
    & \text{ if $n$ is odd. } 
     \end{array} \right. \]                   
\end{lemma}
\begin{proof} As before let $\pi_n = \prod_{i=1}^n (1 - p^{-i})$. We put
  $\beta_{2n} = \prod_{i=1}^n (1 - p^{-2i})$.
  If $n=2k$ or $2k+1$ then
  \begin{align*} 
    \sum_{\substack{\myell=0 \\ \myell \text{ even }}}^n \lambda(n-\myell,n)
    B(\myell) p^\myell &= \sum_{\substack{\myell=0 \\ \myell \text{ even }}}^{2k}
    \frac{\pi_n}{p^{\binom{l+1}{2}} \pi_\myell \beta_{2k - \myell}}
    \cdot \frac{\pi_\myell}{p^{\myell/2} \beta_\myell^2} \cdot p^\myell \\
    &= \frac{\pi_n}{\beta_{2k}^2}
    \sum_{\substack{\myell=0 \\ \myell \text{ even }}}^{2k}
    \frac{\beta_{2k}^2}{p^{\myell^2/2} \beta_{2k-\myell} \beta_\myell^2} \\
    &= \frac{\pi_n}{\beta_{2k}^2}
    \sum_{s=0}^k
     \frac{\beta_{2k}^2}{p^{2(k-s)^2} \beta_{2s} \beta_{2(k-s)}^2}.
  \end{align*}
  The last sum here is $1$, as is seen by taking $(m,n) = (k,k)$ in
  Lemma~\ref{lem:os} and replacing $p$ by $p^2$. This leaves us with
  $\pi_n/\beta_{2k}^2$ which, upon splitting into the cases $n$ even
  and $n$ odd, agrees with the answer in the statement of the lemma.
    
  If $n=2k+1$ or $2k+2$ then
  \begin{align*} 
    \sum_{\substack{\myell=0 \\ \myell \text{ odd }}}^n \lambda(n-\myell,n)
    B(\myell) p^{\myell+1} &= \sum_{\substack{\myell=1 \\ \myell \text{ odd }}}^{2k+1}
    \frac{\pi_n}{p^{\binom{l+1}{2}} \pi_\myell \beta_{2k+1 - \myell}}
    \cdot \frac{\pi_\myell}{p^{(\myell+1)/2} \beta_{\myell-1} \beta_{\myell+1} }
    \cdot p^{\myell+1} \\
     &= \frac{\pi_n}{\beta_{2k} \beta_{2k+2} }
    \sum_{\substack{\myell=1 \\ \myell \text{ odd }}}^{2k+1}
    \frac{\beta_{2k} \beta_{2k+2} }{p^{(\myell^2-1)/2} \beta_{2k+1-\myell}
    \beta_{\myell-1} \beta_{\myell+1}} \\
     &= \frac{\pi_n}{\beta_{2k} \beta_{2k+2}}
    \sum_{s=0}^k
    \frac{\beta_{2k}  \beta_{2k+2}}{p^{2(k+1-s)(k-s)} \beta_{2s}
    \beta_{2(k-s)} \beta_{2(k+1-s)} }.
  \end{align*}
  The last sum here is $1$, as is seen by taking $(m,n) = (k,k+1)$ in
  Lemma~\ref{lem:os} and replacing $p$ by $p^2$.  This leaves us with
  $\pi_n/(\beta_{2k} \beta_{2k+2})$ which, upon splitting into the
  cases $n$ even and $n$ odd, agrees with the answer in the statement
  of the lemma.
\end{proof}

\begin{lemma}
\label{lem:pi_and_f}
Let $\pi_i(\myell,m,n)$ be as defined in Section~\ref{sec:count}.
\begin{enumerate}
\item
 For $i,m \in \{0,1,2\}$ and $n \mygeq i$ we have
  \[ \pi_i(\myell,m,n) = \frac{1}{2}
      \left(1 + \delta_{m1} + \frac{(i-1)(m-1)}{p^{r/2}} \right)
      \lambda(r,n-i) \]
  where $r = 2\myell + m - i$ and $\delta_{m1}$ is the Kronecker delta.
\item Suppose that $f(i,j,n) = \sum_{t=0}^2 (j-1)^t f_t(i,n-i-j)$.
  Then for $i,j \in \{0,1,2\}$ and $n' = n - i -j \mygeq 0$ we have
 \begin{align*}
    \sum_{\myell \mygeq 0}  & \sum_{m=0}^2  \pi_i(\myell,m,n-j)
                  f(j,m,n-2\myell)   
  = \sum_{r=0}^{n'} \lambda(r,n')  f_0(j,n'-r) \\ & + (i-1)
     \sum_{\substack{r=0 \\ r \text{ even }}}^{n'}
     \lambda(r,n')  p^{-r/2} f_1(j,n'-r) +  
     \sum_{\substack{r=0 \\ r+i \text{ even }}}^{n'}  \lambda(r,n') f_2(j,n'-r).  
 \end{align*}
\end{enumerate}
\end{lemma}
\begin{proof}
  (i) This follows from the formulae
  for the $\pi_i(\myell,m,n)$ in Section~\ref{sec:count}. 
  Notice that the term involving $p^{r/2}$ only contributes
  when $i$ and $m$ are even, in which case $r/2$ is an integer. \\
(ii)    Replacing $\myell$ by $(r+i-m)/2$ and using (i)
    the left hand side becomes
    \begin{align*}
    \sum_{\substack{r=0 \\ r+i \text{ even }}}^{n'}
    & \left( \frac{1}{2}\left(1- \frac{i-1}{p^{r/2}}\right)
    \lambda(r,n') f(j,0,n-i-r) \right. \\
    & \left. + \,\, \frac{1}{2}\left(1+ \frac{i-1}{p^{r/2}}\right)
    \lambda(r,n') f(j,2,n+2-i-r) \right) \\
   & \hspace{-2em} + \sum_{\substack{r=0 \\ r+i \text{ odd }}}^{n'} \lambda(r,n')
      f(j,1,n+1-i-r).
    \end{align*}
   Writing $f$ in terms of $f_0$, $f_1$, $f_2$ this becomes 
     \begin{align*}
    \sum_{\substack{r=0 \\ r+i \text{ even }}}^{n'}
    & \left( \frac{1}{2}\left(1- \frac{i-1}{p^{r/2}}\right)
    \lambda(r,n') \big[ f_0(j,n'-r) - f_1(j,n'-r) + f_2(j,n'-r) \big] \right. \\
    & \left. + \,\, \frac{1}{2}\left(1+ \frac{i-1}{p^{r/2}}\right)
    \lambda(r,n') \big[ f_0(j,n'-r) + f_1(j,n'-r) + f_2(j,n'-r) \big] \right) \\
   & \hspace{-2em} + \sum_{\substack{r=0 \\ r+i \text{ odd }}}^{n'} \lambda(r,n')
    f_0(j,n'-r).   
     \end{align*}
     This simplifies to the expression in the statement of the lemma.
     Notice that the sum involving $f_1$ only contributes
     for $i \in \{0,2\}$
     and so the condition ``$r+i$ even'' simplifies to ``$r$ even''.
\end{proof}

The functions $\phi(i,j,n)$ and $\psi(i,j,n)$ were defined in
Definition~\ref{def:phipsi}. The aim of this appendix is to prove
Proposition~\ref{prop:solve-recur} which for convenience we now restate.

\begin{proposition}
\label{prop:phipsi}
Let $i,j \in \{0,1,2\}$ and $n \mygeq i+j$. Then
\begin{equation}
    \label{eqn:phi}
    \phi(i,j,n) = \sum_{\myell \mygeq 0} \sum_{m=0}^2 \pi_i(\myell,m,n-j)
    \phi(j,m,n-2\myell),
  \end{equation}
  and if $n$ is {\bf even} then 
 \begin{equation} 
    \label{eqn:psi}
    \psi(i,j,n) = \sum_{\myell \mygeq 0} \sum_{m=0}^2 \pi_i(\myell,m,n-j)
                  \psi(j,m,n-2\myell).
\end{equation}
\end{proposition}

The condition $n \mygeq i+j$ ensures that $\pi_i(\myell,m,n-j)$ is
defined. It can only be non-zero if $2 \myell + m \myleq n-j$,
equivalently $n - 2 \myell \mygeq j + m$, in which case
$\phi(j,m,n-2\myell)$ and $\psi(j,m,n-2\myell)$ are defined.

\begin{proof}      
  We have $\phi(i,j,n) = (j-1)A(n-i-j) + (i-1)B(n-i-j)$ where $A$ and
  $B$ were defined in~\eqref{eqn:AB}.  It follows that
  $\phi(i,j,n) = \sum_{t=0}^2 (j-1)^t \phi_t(i,n-i-j)$ where
  \[ \phi_0(i,n) = (i-1) B(n), \quad \phi_1(i,n) = A(n),
    \quad \phi_2(i,n) =0. \]
  By Lemma~\ref{lem:pi_and_f}(ii) the right hand side of~\eqref{eqn:phi} is
  \[ (j-1) \sum_{r=0}^{n'} \lambda(r,n') B(n'-r) + (i-1)
    \sum_{\substack{r=0 \\ r \text{ even }}}^{n'} \lambda(r,n')
    p^{-r/2} A(n'-r) \]
  where $n' = n - i - j$.  By Lemma~\ref{lem:id2} the first sum is
  $A(n')$ and the second sum is $B(n')$. This gives
  $(j-1) A(n') + (i-1) B(n') = \phi(i,j,n)$, which is the left hand
  side of~\eqref{eqn:phi} as required.

  We have $\psi(i,j,n) = \sum_{t=0}^2 (j-1)^t \psi_t(i,n-i-j)$ where
\[ \hspace{-2em} \psi_0(i,n) = \frac{p^{2d+1} - p^{\delta_{i1}}}{(p+1)(p^{2d+1}-1)}
  , \quad \psi_1(i,n) = \frac{(i-1)p^{d}(p^2-1)}{(p+1)(p^{2d+1}-1)}, \quad
  \psi_2(i,n) = \frac{p^{2d+1}(p-1)}{(p+1)(p^{2d+1}-1)}, \]
  and $d = \lfloor \frac{n+1}{2} \rfloor$.

  Now suppose that $n$ is even.  If $j \in \{0,2\}$ then by
  Lemma~\ref{lem:pi_and_f}(ii) the right hand side of~\eqref{eqn:psi}
  is
   \[ \frac{1}{p+1} \left( \sum_{\substack{r=0 \\ n'-r \text{ even }}}^{n'} \lambda(r,n') \frac{x-p^r}{p^{n'+1}-p^r}
      \,\,+ \sum_{\substack{r=0 \\ n'-r \text{ odd }}}^{n'} \lambda(r,n') \right)
  \]
  where $n' = n-i-j$ and
  $x = p^{n'+2} + (i-1)(j-1) p^{n'/2} (p^2 - 1)$.  By
  Lemma~\ref{lem:id1} this is equal to $(x-1)/((p+1)(p^{n'+1}-1))$ if
  $i \in \{0,2\}$ and $p/(p+1)$ if $i = 1$. This is equal to
  $\psi(i,j,n)$ as required.
  
  If $j=1$ then by Lemma~\ref{lem:pi_and_f}(ii) the right hand side
  of~\eqref{eqn:psi} is
  \[ \frac{p}{p+1} \left( \sum_{\substack{r=0 \\ n'-r \text{ even }}}^{n'}
      \lambda(r,n') \frac{p^{n'}-p^r}{p^{n'+1}-p^r}
      \,\,+ \sum_{\substack{r=0 \\ n'-r \text{ odd }}}^{n'} \lambda(r,n') \right).
  \]
  By Lemma~\ref{lem:id1} this is equal to
  $(p^{n'+1}-p)/((p+1)(p^{n'+1}-1))$ if $i = 1$ and $1/(p+1)$ if
  $i \in \{0,2\}$. This is equal to $\psi(i,j,n)$ as required.
\end{proof}

\bigskip
\bigskip

Lycka~Drakengren,
\textsc{Trinity College, Cambridge CB3 9DH, UK}\par\nopagebreak
\textit{E-mail address}: \texttt{lyckasbrev@hotmail.com}

\medskip

Tom~Fisher,
\textsc{University of Cambridge, DPMMS, Centre for Mathematical Sciences,
        Wilberforce Road, Cambridge CB3 0WB, UK}\par\nopagebreak
\textit{E-mail address}: \texttt{T.A.Fisher@dpmms.cam.ac.uk}

\end{document}